\documentclass[12pt,twoside]{article}
\usepackage[T1]{fontenc}                        
\usepackage[latin2]{inputenc}             
\usepackage{graphics}                     
\usepackage{psfrag}                       

\usepackage{enumerate}

\usepackage{amsmath}
\usepackage{amsfonts}
\usepackage{amssymb}
\usepackage{amsthm}

\def\titlerunning#1{\gdef\titrun{#1}}
\makeatletter
\def\author#1{\gdef\autrun{\def\and{\unskip, }#1}\gdef\@author{#1}}
\def\address#1{{\def\and{\\\hspace*{18pt}}\renewcommand{\thefootnote}{}%
\footnote {#1}}%
\markboth{\autrun}{\titrun}}
\makeatother
\def\email#1{e-mail: #1}
\def\subjclass#1{{\renewcommand{\thefootnote}{}%
\footnote{\emph{Mathematics Subject Classification (2010):} #1}}}
\def\keywords#1{\par\medskip
\noindent\textbf{Keywords.} #1}

\newtheorem{thm}{Theorem}[section]
\newtheorem{lemma}{Lemma}[thm]
\newtheorem{prop}[thm]{Proposition}
\newtheorem{cor}[thm]{Corollary}

\theoremstyle{definition}
\newtheorem{example}[thm]{Example}

\newtheorem{definition}[thm]{Definition}

\newcommand{\C}{\mathbb{C}}

\newcommand{\R}{\mathbb{R}}
\newcommand{\Z}{\mathbb{Z}}
\newcommand{\V}{\mathcal{V}}
\newcommand{\E}{\mathcal{E}}
\newcommand{\W}{\mathcal{W}}
\newcommand{\A}{\mathcal{A}}
\newcommand{\T}{\overline{T}}

\newcommand{\eq}{\stackrel{\sim}{=}}
\newcommand{\id}{\mathrm{id}}

\newcommand{\Arg}{\mathrm{Arg}}

\newcommand{\St}{\mathrm{St}}

\usepackage{psfrag} 
\usepackage{epic}  
\usepackage{eepic} 

\begin{document}

\title{The Milnor fiber of the singularity $f(x,y) + zg(x,y) = 0$}
\author{Baldur Sigur\dh sson}
\date{}

\maketitle

\titlerunning{Description of a Milnor fiber}
\address{B. Sigur\dh sson: Central European University, Budapest and A. R\'enyi Institute of Mathematics, 1053 Budapest,
Re\'altanoda u. 13-15,  Hungary.; \email{sigurdsson.baldur@renyi.mta.hu}}

\abstract{We give a description of the Milnor fiber and the monodromy of a
singularity of the form $f+zg = 0$ where $f$ and $g$ define plane curves and
have no common components. The description depends only on the topological
type of the two plane curve germs defined by $f$ and $g$. In particular, this
gives a description of the boundary of the Milnor fiber. As a corollary, we
give a simple formula for the monodromy zeta function and the Euler
characteristic of the fiber.}

\section{Introduction}

Let $\Phi:(\C^3,0) \to (\C,0)$, $(x,y,z) \mapsto f(x,y) + zg(x,y)$ be a
functinon germ, where 
$f,g:(\C^2,0) \to (\C,0)$. We will only require that $f$ and $g$ have no common
factors.  This singularity is not isolated; the singular set is the $z$-axis.
We determine the diffeomorphism type of the Milnor fiber $F_\Phi$ in terms of a
simultaneous embedded resolution graph of $f$ and $g$, and some information
about the monodromy.
In particular, we get a description of the boundary $\partial F_\Phi$.

Singularities of this type play an important role in the investigations of
sandwiched singularities, as described in \cite{JongStraten}.

In section \ref{s:hypersurf} we recall some topological properties of
hypersurface
singularities with emphasis on non-isolated singularities and the singular
Milnor fibre. We then recall some properties of plane curve singularities.
Finally, we recall the notion of a $4$ dimensional handlebody and fix
some notation for surgery.

In \ref{s:description}
we construct a subset $T_{f,g}$ of a common embedded resolution of 
$f$ and $g$ from tubular neighbourhoods around some divisors.
We obtain the space $F_{f,g}$ by performing surgery along
certain embedded disks in $T_{f,g}$. This surgery does not change the
homotopy type. Our main theorem states that
$F_{f,g}$ has the same diffeomorphism type as Milnor fibre $F_\Phi$.
Furthermore, $F_{f,g}$ can be decomposed into a union of sets on which
the monodromy can be described completely. As a corollary, we obtain
a simple formula for the monodromy zeta function and the Euler characteristic
$\chi(F_\Phi)$.

Section \ref{s:proof} contains a proof of the main statement \ref{th:themain}
from section \ref{s:description}.

\section{Hypersurface singularities} \label{s:hypersurf}
\subsection{General results}
In this subsection we recall some of the general properties of the Milnor fiber
of a holomorphic germ $f:(\C^{n+1},0)\to (\C,0)$, the monodromy associated to
such a germ, and other invariants related to these two.

Let $f:(\C^{n+1},0) \to (\C,0)$ be a hypersurface singularity, denote by
$B_\delta$ the closed ball with radius $\delta$ around the origin in
$\C^{n+1}$, and by $D_\epsilon$ the closed disk around the origin in $\C$ with 
radius $\epsilon$. $D$ will denote an arbitrary closed disk in the complex
plane. We let $V_f = \{ z\in \C^{n+1} : f(z) = 0 \}$ and
$S_f = \{ z \in \C^{n+1} : \partial f = 0 \}$. The link of $f$ is defined as
$K = V_f \cap \partial B_\delta$ for $0 < \delta \ll 1$.  

The Milnor fiber $F_f$ of $f$ is by definition the fiber
$f^{-1}(\epsilon) \cap B_\delta$ for $0 < \epsilon \ll \delta \ll 1$.
Then $F_f$ is a smooth $2n$ dimensional manifold, and so has the homotopy type
of a CW complex. In \cite{Milnor1}, Milnor proves that $F_f$ is homotopy
equivalent to a finite $n$-dimensional CW-complex.
Moreover, if $s$ is the dimension of the singular locus $S_f$, then $F_f$ is
$(n-s-1)$-connected, as proved in \cite{Kato-Matsumoto}.

Let $E =  f^{-1}(\partial D_\epsilon) \cap B_\delta$. The function
$E \to \partial D_\epsilon$, $z\mapsto f(z)$ is a locally trivial fiber bundle
with fiber $F_f$.  If $T = \{ z \in \partial B_\delta : |f(z)| < \epsilon\}$,
we can define another fiber bundle 
$\partial B_\delta \setminus T \to \partial D_1$, $z \mapsto f(z) / |f(z)|$.
These two fiber bundles are isomorphic. In fact, there is a bundle-isomorphism
$E\to \partial B\setminus T$ which restricts to the identity on $\partial T$.
In particular, we have a diffeomorphism
\begin{equation} \label{eq:2nFiber}
  F_f \eq \{ z \in \partial B_\delta \setminus T : f(z) / |f(z)| = 1 \}.
\end{equation}

The singular fiber of $f$ is defined as 
\[
  F_{f,sing} = \{ z : |z| = \delta,\, |f(z)| > 0,\, f(z) / |f(z)| = 1 \} \cup K.
\]
Usually, $F_{f,sing}$ is not a smooth manifold.
By the description \ref{eq:2nFiber} of $F_f$ we have an inclusion
$\iota : F_f \hookrightarrow F_{f,sing}$.  
If $f$ is an isolated singularity, $\iota$ is a homotopy equivalence, as proved
in \cite{Milnor1}. For non-isolated singularities this does generally not hold.

\subsection{The zeta function of the monodromy}

The monodromy of the Milnor fibration is a diffeomorphism $m_f:F\to F$
with the property
that this bundle is isomorphic to the bundle given by
$F\times I/((p,0)\sim (m_f(p),1)) \to I/(0\sim 1)$, $(p,t)\mapsto t$.
The monodromy is determined by the bundle up to isotopy, and the bundle is
determined up to bundle isomorphism by the monodromy. The monodromy induces
linear isomorphisms $h_i : H_i(F;\C) \to H_i(F;\C)$.

We call the product
\[
  \zeta_f(t) = \prod_{i=0}^\infty \det(I - th_i)^{(-1)^{i+1}}
\]
the zeta function associated with the singularity $f$. This product is well
defined because $F$ is a finite CW complex, and so
$\dim_\C H_*(F;\C) < \infty$. The zeta function behaves multiplicatively in the
following sense:

Let $C$ be a subset of $F$ so that $\dim H_*(C;\C) < \infty$ and $m_f$
restricts to a homeomorphism $m_C : C \to C$. Let us call such a subset good
with respect to $m$.  Then $m_C$ induces a linear automorphism $h_{C,i}$ on
$H_i(C;\C)$ and we define
\[
  \zeta_C(t) = \prod_{i=0}^\infty \det(I - th_{C,i})^{(-1)^{i+1}}.
\]
The following propositions are well known:
\begin{prop} \label{prop:zeta_mult}
Assume that $A,B\subset F$ so that $A, B, A\cap B$
are good subsets of $F$ and the interiors of $A$ and $B$ cover $F$.
Then we have $\zeta_f(t) = \zeta_A (t) \zeta_B (t) \zeta_{A\cap B} (t)^{-1}$.
\qed
\end{prop}

\begin{prop} \label{prop:zeta_char}
We have $\chi(F) = \nu(\zeta_f)$, where $\nu:\C(t)^\times\to\Z$ is the
valuation at infinity.
\qed
\end{prop}

The monodromy $m_f$ can be extended to a homeomorphism
$m_{f,sing}:F_{f,sing} \to F_{f,sing}$, which is called the singular monodromy.
In fact, by defining $F_{f,sing,\theta}$ in the same way as $F_{f,sing}$, only
replacing the condition $f/|f|=1$ by $f/|f|=\theta$, we get a subspace
$\cup_\theta F_{f,sing,\theta} \times \{\theta\} \subset S^{2n+1}\times S^1$.
The projection onto $S^1$ is a locally trivial fiberbundle with fiber
$F_{f,sing}$; its monodromy is the singular monodromy.

%
%
%
%

\subsection{Plane curves} \label{ssec:pl_c}
In the case $n = 1$, $f$ is a plane curve singularity. For a detailed
introduction, see \cite{Wall}.
Write $f = f_1^{\alpha_1} f_2^{\alpha_2} \cdots f_k^{\alpha_k}$
where $f_1,\ldots, f_k$ are the $k$ different irreducible factors of $f$. 
In this case, $K$ is a link in $\partial B_\delta$.
Let $T$ be a tubular neighbourhood around $K$ and $\T$ the corresponding closed
tubular neighbourhood. There exists a projection
$c:\T \to K$ which is a trivial $D$-bundle, this is just the normal bundle
of the link.
Further, write $K = \cup_{i=1}^k K_i$, where
$K_i = \{z \in \partial B_\delta : f_i = 0 \}$, and
$T = \cup_{i=1}^k T_i$, where $T_i$ is the component of $T$ containing $K_i$.
Choosing $\epsilon > 0$ small enough, we can choose
$T = \{ z \in \partial B_\delta : |f(z)| < \epsilon \}$. Then
$\partial F_f \subset \partial \T$. The projection $c$ can be chosen in
such a way that the restriction
$c_i = c|_{F_f\cap \partial \bar T_i}:F_f\cap \partial T_i \to K_i$ 
is a covering map. This map can be described in terms of the embedded
resolution graph of $f$; we recall some of its properties.

Let $\Gamma_f = (\V,\E)$ be the embedded resolution graph of some fixed
embedded resolution of $f$ (see \cite{Wall} for definition and properties).
Here $\V$ is the set of vertices and $\E$ the set
of edges. Write $\V = \W \amalg \A_f$ where $\A_f$ consists of the arrowhead
vertices of $\Gamma$ and $\W$ consists of the nonarrowhead vertices. The
elements of $\A_f$ correspond to the branches of $f$ so there is a natural
correspondence between the arrowhead vertices of $\Gamma$ and the components
of $K$. We will make no distinction between the indices $i=1,\ldots,k$ and
the corresponding $a\in\A_f$.

For each $a \in \A_f$
there exists a unique $w_a \in \W$ so that $(w_a,a) \in \E$. The map $f$ has
multiplicity $\alpha_a$ on $a$, let $m_{w_a}$ be its multiplicity on $w_a$.
Then $F_f \cap T_a$ has $(\alpha_a,m_{w_a})$ components, and restricting $c_a$
to any of these components gives a covering of degree $\alpha_a/(\alpha_a,m_a)$.
The singular fiber $F_{f,sing}$ of $f$ is homeomorphic to the space
$F_f / \sim$ where the equivalence relation $\sim$ is given by
$z_1 \sim z_2$ if and only if $z_1, z_2 \in F_f \cap T_i$ for some $a$, and
$c_a(z_1) = c_a(z_2)$.

The monodromy $m_f:F_f\to F_f$ can be chosen so that it preserves this
equivalence relations, that is, $x_1\sim x_2$ if and only if
$m(x_1) \sim m(x_2)$. Therefore, we get a homeomorphism
$F_{f,sing} \to F_{f,sing}$ induced by the monodromy. This homeomorphism
coincides with the singular monodromy already constructed. Note that
$F_{f,sing} = F_f\cup B$ where both
$B$ and $F_f\cap B$ are homotopically equivalent to the disjoint union of copies
of $S^1$(here, the set $B$ is a disjoint union of sets of the form
$S^1\times R$ where $R$ is a union of segments in the plane with one
endpoint at the origin).
Moreover, these homotopy $S^1$'s contract to actual copies of
oriented $S^1$'s. The singular monodromy $m_{f,sing}$ restricts
to a homeomorphism
$F_f\to F_f$ which coincides with the monodromy $m_f$. Also, $m_{f,sing}$
permutes
the connected components of $B$ and $F_f\cap B$, respecting the orientation
on the first homology. Thus, the induced maps on the homologies of $B$ and
$F_f\cap B$ are zero in degree $>1$, and can be represented by the same
permutation matrix in degrees zero and one. These cancel out
to give
$\zeta_B(t) = \zeta_{F_f\cap B}(t) = 1$, and therefore, by \ref{prop:zeta_mult},
\begin{prop} \label{prop:reg_sing}
If $f:(\C^2,0)\to(\C,0)$ defines a plane curve singularity, then
$\zeta_f(t) = \zeta_{f,sing}(t)$.
\end{prop}

\subsection{Handles and surgery} \label{pre:framings}

We will use handles to describe the Milnor fiber. More precisely, we will use
$4$ dimensional handles of index $2$ in our construction.
Chapter 4 of \cite{GS} gives a good presentation of the necessary theory.

Let $X$ be a $4$-manifold with boundary and
$\phi:(\partial D)\times D \to \partial X$ an
embedding. We obtain a new manifold $X\cup_\phi D\times D$ by taking the
disjoint union $X\amalg(D\times D)$ and then identifying any point
$x\in (\partial D) \times D$ with $\phi(x) \in \partial X$. The map $\phi$
induces an isomorphism between the normal bundles of $(\partial D)\times \{0\}$
in $(\partial D) \times D$ and $\phi((\partial D)\times \{0\})$ in $\partial X$.
Since $(\partial D)\times \{0\} \subset (\partial D)\times D$ already comes
with a canonical framing, this isomorphism can be specified by a framing on
$\phi((\partial D)\times \{0\})$. The diffeomorphism type of the resulting
manifold is determined by the following data (see for example \cite{GS}):

\begin{itemize}
\item
	The embedding $\phi|_{(\partial D)\times \{0\}}$ of
	$(\partial D)\times \{0\} \eq S^1$ into $\partial X$.
\item
	The framing of the normal bundle of $\phi|_{(\partial D)\times \{0\}}$.
\end{itemize}

We will now fix some notation for surgery along embedded
disks. We will assume that all maps respect orientation when appropriate.
Let $X$ be an oriented $4$ dimensional manifold with boundary and
$\iota:\bar D\hookrightarrow X$
an embedding of the closed disk. We assume that the boundary
$\partial\bar D$ is
embedded into the boundary $\partial X$, and that $\iota(\bar D)$
is transversal to
$\partial X$. We can find a parametrisation
$\psi:\bar D\times \bar D \to X$ of a closed tubular neighbourhood of
$\iota(\bar D)$ so that $\psi(0,z) = \iota(z)$, and
$\psi|_{\bar D\times \partial \bar D}$ is a parametrisation of a tubular
neighbourhood of $\iota(\partial\bar D)\subset\partial X$. Define
$X' = X \setminus \psi(D\times\bar D)$. For $k\in\Z$ let
$X_{\iota,k} = X'\amalg_{t_k} \bar D\times \bar D$, where the glueing map
$t_k:\bar D\times\partial\bar D\to X'$ is given by
$t_k(x,y) = \psi(x,x^ky)$.

\begin{definition}
We call $X_{\iota,k}$ constructed above the $k$-th twist of $X$ along
$\iota(\bar D)$.
\end{definition}

Note that $X_{\iota,k}$ is obtained by thinking of $\psi(\bar D\times \bar D)$
as a handle, removing it, and then attaching it again via a different
glueing map. This construction is very similar to Dehn surgery. In fact,
$\partial X_{\iota,k}$ is nothing else than $\partial X$, to which
a Dehn surgery with coefficient $1/k$ has been applied along
$\iota(\partial \bar D)$.

\section{Description of the fiber} \label{s:description}

Let $f, g:(\C^2,0) \to (\C,0)$ be any plane curve singularities without common
factors and define
\[
  \Phi(x,y,z) = f(x,y) + zg(x,y).
\]
Consider a fixed common embedded resolution $\phi:V\to \C^2$ of $f$ and $g$.
The resolution graph of this embedded resolution will be denoted by $\Gamma$.
Denote its set of vertices by $\V(\Gamma)$ and the set of edges by $\E(\Gamma)$.
We write $\V(\Gamma) = \W(\Gamma) \amalg \A(\Gamma)$ where $\W(\Gamma)$ is the
set of non-arrowhead vertices and $\A(\Gamma)$ the set of arrowhead vertices.
We decompose $\A(\Gamma)$ further as
$\A(\Gamma) = \A_f(\Gamma) \amalg \A_g(\Gamma)$, where the elements of
$\A_f(\Gamma)$ and $\A_g(\Gamma)$ correspond to components of the strict
transform of $f$ and $g$ respectively. A vertex $v\in\V(\Gamma)$ corresponds to
a component $E_v$ of the exceptional divisor $\phi^{-1}(0)$, or the strict
transform of $f$ or $g$. In each case, we denote by $m_v$ the multiplicity of
$f$ on $E_v$, and $l_v$ the multipilcity of $g$ on $E_v$. In particular,
$m_v = 0$ if and only if $v\in \A_g$ and $l_v = 0$ if and only if $v\in A_f$.

Let $f' = f\circ \phi$, $g' = g\circ \phi$ and 
$F'_f = (f')^{-1}(\epsilon) \cap \phi^{-1}(B_\delta) = \phi^{-1}(F_f)$.
The map $V\setminus \phi^{-1}(0) \to \C^2\setminus \{(0,0)\}$, $r\to \phi(r)$ is
a diffeomorphism. In particular, it restricts to a diffeomorphism
$F'_f \to F_f$.

We have a map $\phi\times \id_\C : V \times \C \to \C^3$ which restricts to a
diffeomorphism
$(V\setminus \phi^{-1}(0))\times \C \to \C^3 \setminus \{(0,0,z) : z\in \C\}$.
We set
$\Phi' = \Phi\circ(\phi\times \id_\C)$, and $F' = (\phi\times \id_\C)^{-1}(F)$. Clearly, $F'$ is diffeomorphic to $F$.

For each $w\in \W$, choose a small tubular neigbourhood $T_w$ around $E_w$ in
$V$ and a map $b_w : T_w \to E_w$ which is a smooth open disk bundle. Denote
by $\T_w$ the corresponding closed tubular neighbourhood. These can be chosen
so that they satisfy the following property:

If $w,w' \in \W$ and $(w,w') \in \E$, then we have
$b_w^{-1} (E_w \cap E_{w'}) = E_{w'} \cap T_w$ and
$b_w^{-1}(E_w \cap T_{w'}) = T_w \cap T_{w'}$. If $w\in \W$, $a\in \A$ and
$(w,a) \in \E$, then $b_w^{-1}(E_w\cap E_a) = E_a \cap T_w$.  Then the set
$T = \cup_{w\in\W} T_w$ is the plumbed 4-manifold with plumbing graph $\Gamma$.

If $w,w'\in \W(\Gamma)$ and $e = (w,w') \in \E(\Gamma)$, then we let
$T_e = T_w \cap T_{w'}$. If $w\in \W(\Gamma)$ and $a\in\A(\Gamma)$ so that
$e = (w,a) \in \E(\Gamma)$, then we pick a small disk-shaped neighbourhood
$U_a$ in $E_w$ around $E_w \cap E_a$ and let $T_a = T_e  = b_w^{-1}(U_a)$.
Then $T_a$ is a tubular neighbourhood around $E_a$ in $T$.

The Milnor fiber $F_\Phi$ can be described in terms of the embedded resolution
graph $\Gamma$, with the additional arrowhead vertices, and all vertices
decorated by the multiplicities of $f'$ and $g'$. This description will depend
on which of the
two functions $f'$ and $g'$ has higher multiplicities on the exceptional
divisors. The following definition makes this precise.
\begin{definition}
\begin{itemize}
\item
Let $\W_1 = \{w\in\W(\Gamma) : m_w \leq l_w \}$ and $\W_2 = \W\setminus \W_1$.
Let $\Gamma_i$ be the subgraph of $\Gamma$ generated by the set $\W_i$.
Define $T_i = \cup_{w\in\W_i} T_w$. 
\item
Let $\A_{f,i} = \{ a\in\A_f : w_a \in\W_i\}$ and
$T_{f,i} = \cup_{a\in\A_{f,i}} T_a$.
Repeat this with $f$ replaced by $g$.
\item
Choose a small $\epsilon > 0$ and let $T_{\epsilon}$ be a small tubular
neighbourhood around $f'^{-1}(\epsilon) \cap \T$.
\item
Let $T'$ be a small tubular neighbourhood around the exceptional divisor
inside $T$. This is chosen after choosing $\epsilon$. In particular,
$\T'\cap \T_\epsilon = \emptyset$.
\item
Let $\T_{f,g} = [\T_{f,1}\setminus T'] \cup \T_\epsilon \cup
[\T_2 \setminus (T' \cup T_{g,2})]$, where $\overline{\phantom{T}}$ denotes
closure.
\item
Let $T'_g$ be a tubular neighbourhood around the strict transform of $g$,
chosen small with respect to the above.
\end{itemize}
\end{definition}

\begin{definition}
We define $F_{f,g}$ to be a twisting of $T_{f,g}$ along the strict
transform of $g$. More
precisely, for any $a\in\A_g$, the set $E_a\cap T_{f,g}$ is a union of
$m_w$ disks embedded in $T_{f,g}$ as in \ref{pre:framings}, where $w\in\W$ so
that $(a,w)\in\E$. Take the $l_a$-th twist along each of these disks.
\end{definition}

\begin{thm} \label{th:themain}
The Milnor fibre $F_\Phi$ is diffeomorphic to the space $F_{f,g}$ constructed
above.
The monodromy can be chosen to satisfy the following
\begin{itemize}
\item
The set $\T_{f,1}\setminus T'$ is invariant under $m_\Phi$ and the
restriction is homotopic to the identity.
\item
We have $m_\Phi|_{F_f} = m_f$
\item
The set $T_2\setminus(T'\cup\T_{g,2})$ is invariant under $m_\Phi$
and the restriction is homotopic to the identity.
\item
For any $a\in\A_{g,2}$, the monodromy $m_\Phi$ permutes the $m_{w_a}$ handles
corresponding to $a$ cyclically.
\end{itemize}
\end{thm}

\begin{cor} \label{cor:eulerchar}
\begin{enumerate}[(i)]
\item
The Euler characteristic of $F$ is given by the formula
\[
  \chi(F_\Phi) = \sum_{w\in\W_1} m_w(2-\delta_{w,f}) 
      + \sum_{a\in\A_{g,2}} m_{w_a}.
\]

\item
The zeta function associated to $\Phi$ is given by the formula
\begin{equation} \label{cor:eulerchar:2}
  \zeta_\Phi(t) =
    \left(\prod_{w\in\W_1}     (1-t^{m_w}    )^{\delta_{w,f}-2} \right)
    \left(\prod_{a\in\A_{g,2}} (1-t^{m_{w_a}})^{-1}             \right)
\end{equation}
where $\delta_{w,f}$ is the number of vertices in $\W\cup\A_f$ connected to $w$
by an edge.
\end{enumerate}
\end{cor}
\begin{proof}
By \ref{prop:zeta_char}, it is enough to prove \ref{cor:eulerchar:2}.

Consider first the action of $m_\Phi$ on
$\T_{f,g} \setminus T_2$. Using a similar argument as in
\ref{ssec:pl_c}, we see that the zeta function of the restriction is the
same as that of the restriction to $F_f\cap \T_1 \setminus T_2$. An A'Campo
type argument shows that this zeta function is
\begin{equation} \label{eq:zeta_first}
  \prod_{w\in\V(\Gamma_1)} (1 - t^{m_w})^{\delta_{w,f}-2}.
\end{equation}

Consider now the set $\T_2 \cap \T_{f,g}$. It has the homotopy type of a
3-manifold with some solid tori removed. In particular,
$\chi(\T_2 \setminus T') = 0$.
We can now use the same proof as that of \ref{cor:main:W2} to see that
the zeta function of the restriction to $T_{f,g}\cap\T_2$
is $(\prod_{a\in\A_{g,2}} (1-t^{m_{w_a}})^{-1})$.

The intersection $(\T_{f,g} \setminus T_2) \cap (\T_{f,g}\cap\T_2$
is a disjoint union of circles which are cyclically permuted
by the monodromy. The zeta function of the monodromy restricted to these
circles is therefore $1$.

Finally, using \ref{prop:zeta_mult}, we can glue these zeta functions together
to get \ref{cor:eulerchar:2}.
\end{proof}

\begin{cor} \label{cor:main}
\begin{enumerate}[(i)]
\item \label{cor:main:W1}
	If $m_w \leq l_w$ for all $w\in\W(\Gamma)$, then $F$ and $F_{f,sing}$
	have the same homotopy type and $\zeta_\Phi = \zeta_f$.
\item \label{cor:main:W2}
If $m_w > l_w$ for all $w\in\W(\Gamma)$, then $F$ has the same homotopy
type as $\vee_{m-1} S^2$, where $m = \sum_{a\in \A_g(\Gamma_2)} m_{w_a}$.
The zeta function is given by 
$\zeta_\Phi(t) = \prod_{a\in\A_g} (1-t^{m_{w_a}})$.
\end{enumerate}
\end{cor}
\begin{proof}
In the case of \ref{cor:main:W1} we have $T_2 = \emptyset$. Twisting the
handles corresponding to
elements $a\in\A_g$ does not alter the homotopy type. 
Therefore, $F_\Phi$ has, by \ref{th:themain}, the same homotopy type as
$\T_\epsilon\cup\T_{f,1}$. Homotopically, this space is the same as
$F_f$, where we have glued the boundary components to some circles. This can
easily be seen as the same construction of $F_{f,sing}$.
Furthermore, this homotopy equivalence can be seen as invariant under
the actions of $m_\Phi$ and $m_f$, proving $\zeta_\Phi = \zeta_{f,sing}$
and thus $\zeta_\Phi = \zeta_f$ by \ref{prop:reg_sing}. For a second
proof of this statement, one may compare A'Campo's formula for
$\zeta_f$ with \ref{cor:eulerchar}.

In the case of \ref{cor:main:W2}, $\A_{f,1} = \emptyset$. We have
$A = \T_2 \setminus (T'\cup T_{g,2})$ and $\T_2 = \T$. Also, $\T\setminus \T'$
is homotopically just $\partial \T = S^3$ because the graph $\Gamma$ describes
a modification of the smooth germ $(\C^2,0)$. In fact, $\T\setminus T'$
is a collar neighbourhood around $\partial \T$, so
$\T\setminus T' = S^3\times I$. Furthermore, for $a\in\A_{g}$, the pair
$(\T\setminus T', \T_a \cap (\T\setminus T'))$ is isomorphic to the pair
$(S^3\times I, S\times I)$ where $S\subset S^3$ is a solid torus.
Therefore, $A$ is homotopically
$S^3$ with some solid tori removed, one for each element of $\A_g$. What's
more, the attaching spheres of the handles are meridians of these tori.
But removing a solid torus from a $3$ manifold and adding $m$ handles attached
to meridians is equivalent to removing $m$ spheres from the original
manifold. This gives the same as $\vee_{m-1} S^2$.

The statment about $\zeta_\Phi$ follows from \ref{cor:main}
%
\end{proof}

\begin{example}
Let $f(x,y) = x^d$ and $g = y^d$ where $d \geq 2$. Then we can choose the
resolution $V$ so that $\V$
has a single element, say $\V = \{v\}$. Then $m_v = l_v = d$, so we can
apply \ref{cor:main}a.
The Milnor fiber $F$ associated to $\Phi$ has the same homotopy type as
$F_{f,sing}$, which is up to 
homotopy a bouquet of $d-1$ two-spheres. Note that in spite of this, $\Phi$
is not isolated.
The zeta function of this singularity is $\zeta(t) = t^d-1$.

\end{example}

\section{Proof of theorem \ref{th:themain}} \label{s:proof}

To prove theorem \ref{th:themain} we project the embedded resolution
$V\times \C \to \C^3$ down to $V$, and study the image of the fiber $F'_\Phi$.
Denote the projection by $p$. Choose $r\in V$ with the property that
$g'(r) \neq 0$. Assume further that there exists a $z\in\C$ such that
$\Phi'(r,z) = \epsilon$. We can solve this equation for $z$, namely
\[
  z = \frac{\epsilon - f'(r)}{g'(r)}.
\]
This means that $p$ restricts to an injection
$F'_\Phi \setminus (\St_g\times\C) \to V$, where $\St_g$ is the strict transform
of $g$. Define a function
$Z:V\setminus \St_g \to \C$ by $Z(r) = (\epsilon - f'(r))/g'(r)$.
Instead of using the standard ball $B_6\subset \C^3$, we can assume that 
$F'_\phi$ is simply defined as the set of pairs $(r,z)\in V\times \C$ where
$r\in T$ and $|z| \leq \delta$. This way, we get a diffeomorphism
$F'_\Phi \setminus \St_g\times\C \to X$ where
\[
  X = \{r\in V \setminus \St_g : |Z(r)| \leq \delta\}.
\]
We obtain a description of $F_\Phi = F'_\Phi$ by considering the sets
$F'_\Phi \cap p^{-1}(\T\setminus T_g)$ and $F'_\Phi\cap p^{-1}(\T_g)$, and
how they glue together along their intersection.

\begin{thm} \label{th:reform}
The following items determine the Milnor fibre and the monodromy.
\begin{enumerate}[(i)]
\item \label{th:reform_strict_f}
	Let $e=(a,w)\in\E$ where $a\in\A_{f,1}$ and $w\in\W$.
	There is a diffeomorphism between $p(F'_\Phi)\cap\T_e$
	and $\T_e\setminus T'$ inducing identity on $F'_f\cap \partial \T_e$
	and its normal bundle in $\partial \T_e$.

	The set $F'_\Phi\cap p^{-1}(\T_e)$ is invariant under the monodromy,
	up to homotopy the monodromy action is trivial on this set.
\item \label{th:reform_T_1}
	The set $p(F'_\Phi)\cap[\T_1 \setminus (T_{f}\cup T'_{g} \cup T_2)]$
	is a tubular neighborhood around 
	$F'_f\cap[\T_1 \setminus (T_{f}\cup T'_{g} \cup T_2)]$ in
	$T_1 \setminus (T_{f}\cup T'_{g} \cup T_2)$.

	The set $F'_\Phi\cap p^{-1}[\T_1 \setminus (T_{f}\cup T'_{g} \cup T_2)]$
	is invariant under the monodromy. It can be chosen to coincide with
	$m_f$ on the subset
	$F'_f\cap [\T_1 \setminus (T_{f}\cup T'_{g} \cup T_2)]$ which is a
	strong homotopy retract.
\item \label{th:reform_T_2}
	There is a diffeomorphism between $p(F'_\Phi)\cap \T_2\setminus T'_g$
	and $\T_2\setminus (T'\cup T'_g)$ inducing identity on
	$F'_f\cap\partial(T'\cup T'_g)$ and its normal bundle in
	$\partial(T'\cup T'_g)$.
	This set is invariant under the monodromy; it's action is trivial up
	to homotopy.
\item \label{th:reform_strict_g}
	Let $e=(a,w)\in\E$ where $a\in\A_g$ and $w\in\W$.
	The set $p^{-1}(\T'_{g})\cap F'_{\Phi}$ is a disjoint union of $m_w$
	$4$ dimensional $2$-handles glued to the manifold
	$p(F'_\Phi) \setminus T'_g$. The attaching spheres are those boundary
	components of $F'_f\cap (\T\setminus T'_e)$ which are in $\T'_e$.
	The normal bundle of the attaching spheres has a canonical
	trivialisation since each component is the boundary of a disk in
	$T'_e$. The handles are attached with the $(-l_a)$-th framing.
	These handles are invariant under the monodromy, its action permutes
	them cyclically.
\end{enumerate}
\end{thm}

\begin{proof}[Proof of \ref{th:reform_strict_f}]
We can choose coordinates around the point $E_w\cap E_a$ so that
$f'(u,v) = u^m v^n$ and $g(u,v) = u^l$, where $m=m_w$, $n=m_a$ and $l=l_w$.
We can also suppose that $\T_e = \{(u,v) : |u|,|v| \leq \rho \}$ where
$\rho$ is some number so that $\epsilon\ll\rho$.
By choices made, we have $m\leq l$ and $n>0$.

Consider the space $\tilde T_e = \{(u,\tilde v): |u|,|\tilde v|^{1/n} < \rho\}$
and the map $\pi_e:\T_e\to \tilde T_e$ given by
$(u,v) \mapsto (u,\tilde v) = (u,v^n)$. We have then maps
\[
  Z_e(u, v) = \frac{u^m v^n - \epsilon}{u^l}, \quad
  \tilde Z_e(u,\tilde v) = \frac{u^m \tilde v - \epsilon}{u^l}
\]
satisfying $\tilde Z_e \circ \pi_e = Z_e$.
The function $|\tilde Z_e|^2$ has the divisor $u^m \tilde v = \epsilon$ as
a nondegenerate critical manifold of index $0$. This holds on
$\tilde T_e$ as well as $\partial \tilde T_e$. We will show that
$|\tilde Z_e|^2$ has no other critical manifolds (in the interior or the
boundary) in the preimage $|\tilde Z_e|^2 \leq \delta^2$. This will show that
the set $\tilde F_e = \pi_e(p(F_\Phi)\cap T_e)$ is a tubular
neighbourhood around the submanifold given by $u^m\tilde v = \epsilon$,
that is, $\pi_e(F_f)$. Note first that the coordinate $u$ takes
nonzero values on $\tilde F_e$, since $Z_e$ has a pole along the
exceptional divisor. We have then
$\partial_{\tilde v} Z_e(u,\tilde v) = u^{m-l} \neq 0$ on $\tilde F_e$.
This shows that $|\tilde Z_e|^2$ has no critical points in the interior
$\tilde T_e$, nor on the part of the boundary given by $|u|=\rho$.
For the rest of the boundary, we will show that if $|\tilde v|=\rho^n$,
then $\partial_u Z_e \neq 0$. But we have
\[
  \partial_u \tilde Z_e = (m-l)u^{m-l-1}\tilde v + lu^{l-1}\epsilon
    = ((m-l)u^m \tilde v + \epsilon l) u^{l-1}.
\]
If $m = l$, then this shows that the partial derivative does not vanish.
Assuming $m<l$ we find that $\tilde Z_e(u,\tilde v) = 0$ implies
$u^m = -\epsilon l / ((m-l)\tilde v)$. This implies
\[
  |\tilde Z_e(u,\tilde v)|
    = \frac{|-\frac{\epsilon l}{(m-l)\tilde v} - \epsilon|}
	    {|-\frac{\epsilon l}{(m-l)\tilde v}|^{l/m}}
    = \left| \frac{l}{(m-l)\rho^n}-1 \right|
	  \left| \frac{(m-l)\rho}{l}\right|^{-l/m}
	  \epsilon^{1-l/m},
\]
so that $|Z_e(u,\tilde v)|$ is huge, since $\epsilon$ is small and $l>m$.
In particular, $|\tilde Z_e| > \delta$.

We have now showed that $\tilde F_e$ is a tubular neighbourhood around
the divisor $u^m \tilde v = \epsilon$. But the same is true about the set
$\tilde T_e \setminus \pi_e(T')$. Thus, we have a diffeomorphism
$\tilde\psi_e:\tilde F_e \to \tilde T_e \setminus \pi_e(T')$ and we can assume
that $\tilde\psi_e$ equals the identity on a small neighbourhood around
the divisor $u_m\tilde v = \epsilon$.
Now, the set $\{\tilde v = 0\} \cap \tilde F_e$ is an annulus given by
$|u| \geq |\epsilon/\delta|^{1/l}$. One can now easily see that the
map $\tilde\psi_e$ can also be chosen to map this annulus into
$\pi_e(E_a) = \{\tilde v = 0\}$.
Finally, by considering the symmetries of $\tilde F_e$ and $\tilde T_e$,
one can assume that $\tilde\psi_e$ commutes with multiplying $\tilde v$
by a primitive $n$th root of unity. This, combined with the fact that
$F_e = \pi_e^{-1}(\tilde F_e)$, shows that $\tilde \psi_e$ transfers
to a diffeomorphism $\psi_e:F_e\to T_e$.
\end{proof}

\begin{lemma} \label{lem:T_1_fibers}
The map $F'_\Phi \cap [(\T_1\setminus(T_f\cup T_2))\times\C] \to\bar D_\delta$,
$(r,z)\mapsto z$
is proper, with surjective derivative everywhere. The same holds for its
restriction to the boundary.
\end{lemma}
\begin{proof}
The map is proper, since its domain is compact. The surjectivity of the
derivative requires more attention:

Let $(r_0,z_0) \in F'_\Phi \cap [(\T_1\setminus(T_f\cup T_2))\times\C]$. Then,
we have three cases: in the first, there is a unique $w\in\W_1$ so that
$r_0\in T_w$. Secondly, there might be exactly
two elements $w,w'\in\W_1$ such that $r_0\in T_w\cap T_{w'}$. Thirdly, we might
have $r_0\in T_e$ for some $e=(w,a)\in\E$ for some $w\in\W_1$ and
$a\in\A_{g,1}$.
In any case, we can find a coordinates $u,v$ in a neighbourhood $U$
around $r_0$ in $V$ such that
$f'(u,v) = u^m v^l$ and $g'(u,v) = \alpha u^l v^k$ for $m=m_w$, $l=l_w$
and some non-vanishing function $\alpha:U\to\C$.
We have $m \leq l$, and one out of three, depending on the cases above:
$n=k=0$, $n\leq k$ or $n=0$,$k>0$. In any case, we have $n\leq k$.

By the inverse function theorem, the map $F'_\Phi\cap U \to \C^2$,
$(u,v,z) \to (v,z)$ is a coordinate chart, provided that
$\partial_u\Phi' \neq 0$ on $F'_\Phi\cap U$. We have
\[
\begin{split}
  \partial_u\Phi'(u,v,z) &= \partial_u(u^m v^n + z\alpha u^l v^k)
     = mu^{m-1}v^n + z((\partial_u\alpha) u^lv^k + \alpha l u^{l-1}v^k) \\
    &= u^{m-1} v^n(m+zu^{l-m} v^{k-n} (\partial_u\alpha u + \alpha l)).
\end{split}
\]
The function $u^{l-m} v^{k-n} (\partial_u\alpha u + \alpha l)$ is continuous,
and therefore bounded on $U$ (we can assume that $U$ is relatively compact).
Since $|z|\leq\delta$, we get
\[
  |zu^{l-m} v^{k-n} (\partial_u\alpha u + \alpha l)| \ll m.
\]
proving that $\partial_u \Phi' \neq 0$ on $F'_\Phi \cap (U\times\C)$.
Therefore, the function $z$ is a part of a coordinate system around
$(r_0,z_0)$. In particular, its derivative is surjective.

For the last statement, the same reasoning applies; the equation
$\partial_u \Phi' \neq 0$ implies that $z$ (as two real variables) gives part
of a coordinate system on the boundary. We omit the details.
\end{proof}

\begin{proof}[Proof of \ref{th:reform_T_1}]
The argument in the proof of \ref{lem:T_1_fibers} can be transferred directly
to the boundary components $p(F'_\Phi)\cap \partial \T_{g,1}$. We can therefore
use Ehresmann's fibration
theorem to get that the restriction of $Z$ to the set
$p(F'_\Phi)\cap[\T_1\setminus(T_f\cup T_g \cup T_2)]$ is a locally trivial
fibration over $D_\delta$. Since $D_\delta$ is contractible, this fibration
is trivial. The fiber over $0\in D_\delta$ is simply
$F'_f\cap [\T_1\setminus(T_f\cup T_g \cup T_2)]$.
Therefore, the set $p(F'_\Phi)\cap[\T_1\setminus(T_f\cup T_g \cup T_2)]$ is
a product $F'_f\cap [\T_1\setminus(T_f\cup T_g \cup T_2)]\times D$.
This proves the statement.
\end{proof}

\begin{lemma}\label{lem:g/f}
We may assume that the inequality $|g/f| < \delta/2$ holds in
$\T_2$.
\end{lemma}
\begin{proof}
Let $x,y$ be some generically chosen coordinates on $\C^2$. In $\T_1$ (in all
of $V$, for that matter) we have $|x'|,|y'| \leq \delta$, where $x',y'$ are
the pullbacks of $x,y$. In $T_2$, the function $f'/g'$ vanishes along
$E\cap\T_2$ by definition of $\W_2$. Since $x'$ and $y'$ vanish with order one
along $E$ we have $|f'/g'| \leq C \|(x',y')\|$ in $\T_2$ for some $C>0$.
Multiplying $f$ with $C^{-1}$, however, gives an equivalent singularity
because the germ $C^{-1}f + zg = C^{-1}(f+Czg)$ is equivalent with the germ
$f+zg$ via the coordinate change $(x,y,z) \leftrightarrow (x,y,Cz)$.
\end{proof}

\begin{proof}[Proof of \ref{th:reform_T_2}]
We start by investigating the intersection of $p(F'_\Phi)$ with the smaller set
$\T_2\setminus(T_1\cup T_g)$. The remaining parts will be considered separately.

As before, we have
\[
  p(F'_\Phi)\cap \T_2\setminus(T_1\cup T_g)
    = \{r\in\T_2\setminus(T_1\cup T_g) : |Z(r)| \leq \delta \}.
\]
We will start by showing that $|Z|^{-1}$ is a Morse-Bott function
in $\T_2\setminus(T_1\cup T_g)$ which defines a small tubular neighbourhood
around the exceptional divisor. More precisely, let
\[
  N = \{r\in\T_2\setminus(T_1\cup T_g) : |Z(r)|^{-1} < \delta^{-1} \}.
\]
We will prove that $N$ is a tubular neighbourhood around the exceptional
divisor in $\T_2\setminus(T_1\cup T_g)$, and that it can be made arbitrarily
small by shrinking $\epsilon$. The restriction of $g'$ to
$\T_2\setminus(T_1\cup T_g)$ is a holomorphic function vanishing exactly on
the exceptional divisor. Therefore, the set
$\{ r\in\T_2\setminus(T_1\cup T_g) : |g'(r)| \geq 2\epsilon/\delta\}$
is the complement of a small neighbourhood around the exceptional divisor.
If $r\in\T_2\setminus(T_1\cup T_g)$ satisfies $|g'(r)| \geq 2\epsilon/\delta$,
we get
\[
  |Z(r)| = \left|\frac{f'(r)-\epsilon}{g'(r)}\right|
    \leq \left|\frac{\epsilon}{g'(r)}\right|
      + \left|\frac{f'(r)}{g'(r)}\right|.
\]
By the choice of $r$, we have $\epsilon/|g'(r)| < \delta/2$. By \ref{lem:g/f},
we also have $|f'(r)/g'(r)| \leq \delta/2$. Therefore, we get
$|Z(r)| \leq \delta$. We have proven
\[
  N \subset N' := \{r\in V : |g'(r)| \leq 2\epsilon/\delta\}
      \setminus (T_1\cup T_g).
\]
The set $N'$ above can be made arbitrarily small, as a neighbourhood around
the exceptional divisor. To show that $N$ is a tubular neighbourhood, we will
prove that the derivative of $Z$ does not vanish in $N'$ outside the exceptional
divisor. Choose coordinates $u,v$ around $r\in V$ such
that $f'(u,v) = u^m v^l$ and $g'(u,v) = \alpha u^l v^k$ where $m=m_w$, $l=l_w$
for some $w\in\W_2$ for which $r\in \T_w$ and either there is a $w'\in\W_2$
so that $n=m_{w'}$ and $k=l_{w'}$, or $n=l=0$. In any case, we have $m>l$
and $n\geq k$.
We calculate:
\[
  \partial_u Z(r) = \frac{\partial}{\partial u}
      \frac{u^m v^n - \epsilon}{\alpha u^l v^k}
    = \frac{mu^{m-1}v^n \alpha u^l v^k
        - (u^m v^n - \epsilon)( (\partial_u\alpha)u^l - \alpha lu^{l-1} )v^k}
	  {(\alpha u^l v^k)^2}.
\]
Simplifying, we get $\partial_u Z(r) \neq 0$ if and only if
\begin{equation} \label{eq:pf_of_iii_a}
  m u^m v^n \alpha - (u^m v^n - \epsilon)((\partial_u\alpha)u - \alpha l)
    \neq 0.
\end{equation}
By assumption, we have $|Z(r)| > \delta$, that is,
$|u^m v^n - \epsilon| > \delta |\alpha u^l v^k|$.
Thus, we prove \ref{eq:pf_of_iii_a} by showing that
\begin{equation} \label{eq:pf_of_iii_b}
  |m u^{m-l} v^{n-k}| < \delta|(\partial_u\alpha)u - \alpha l|.
\end{equation}
The number $|(\partial_u\alpha)u - \alpha l|$ is bounded below independent of
$\delta$ and $\epsilon$, because $|u|$ is small with respect to $\alpha l$,
which is bounded below. The functions $|u^{m-l} v^{n-k}|$ and $g'(u,v)$ have
the same zero set, thus there is a $C,\gamma\in\R_+$ so that
$|mu^{m-n} v^{l-k}| < C |g'(u,v)|^\gamma
\leq C(2\epsilon/\delta)^\gamma \ll \delta$. This proves \ref{eq:pf_of_iii_b}.
Hence, the set $N$ is a tubular neighborhood around the exceptional divisor
in $\T_2\setminus(T_2\cup T_g)$.

Now consider an edge $e=(w_1, w_2) \in \E$ where $w_i\in\W_i$. We want to prove
that there is a diffeomorphism between$T_e\cap p(F'_\Phi)$ and
$\T_e\setminus T'$
fixing the intersection $F'_f\cap \partial \T_e$ and its normal bundle
inside $\partial \T_e$.

Consider coordinates $u,v$ on $T_e$ so that $f=u^m v^n$ and $g=u^l v^k$.
Let $\tau_1, \tau_2\in\C$ with $|\tau_1| = 1$. Then the set
$\{(u,v) \in\T_e :\Arg(u) = \tau_1,\, \Arg(v) = \tau_2,\,|Z(u,v)|\leq \delta\}$
is a disk. In fact:
\begin{itemize}
\item
	If $\tau_1^m\tau_2^n \neq 1$, then $|Z|^2$ restricts to a Morse
	function on the manifold
	$\{(u,v):\Arg(u)=\tau_1,\,\Arg(v) = \tau_2\}$. There are no critical
	points on the interior. Restricting $Z$ to the boundary of this
	submanifold, we get exactly one critical point with index zero
	and at most one with index one.
\item
	If $\tau_1^m\tau_2^n = 1$, then$|Z|^2$ restricts to a Morse-Bott
	function on the submanifold
	$\{(u,v):\Arg(u)=\tau_1,\,\Arg(v) = \tau_2\}$, the critical set being
	the intersection with $F'_f$.
\end{itemize}
Proving these two statements is a simple exercise, it boils down to showing that
certain partial derivatives do not vanish. The results show that each fiber
of the argument map $(\tau_1, \tau_2)$ is abstractly a disk. One is therefore
free to choose a diffeomorphism from this disk to the set of points $(u,v)$
where the argument of each coordinate is fixed, to the set of points with
corresponding arguments in $\T_e\subset T'$. This can be done in such a way that
we get a diffeomorphism with the desired properties.

The last thing we need to consider is the set
$p(F'_\Phi)\cap (\T_{g,2}\setminus T'_{g,2})$. Let $a$ be in $\A_{g,2}$. We
have local coordinates $u,v$ on $T_a$ so that $f = u^m$ and $g = u^l v^k$,
where $m=m_{w_a}$, $l=l_{m_w}$ and $k=m_a$. The set
$p(F'_\Phi)\cap (\T_a\setminus T'_a)$ can be given by equations
$|Z| \leq \delta$ and $|v| \geq \eta$ for some $\eta \ll \epsilon$, that is,
\[
  \left\{ (u,v) \in\T_a :
         \left|\frac{u^m - \epsilon}{u^l v^k}\right| \leq \delta,\,
         |v| \geq \eta \right\}
\]
We proved already, that the intersection $p(F'_\Phi)\cap \{|v| = \rho_a\}$ is
the complement of a tubular neighbourhood around the exceptional divisor
in the set $\{|v| = \rho_a\}$. Take a point 
$(u_0,v_0)\in p(F'_\Phi)\cap \T_a\setminus T'_a$.
From the formula $Z(u,v) = (u^m-\epsilon)/(u^l v^k)$ we see that
the segment between $(u_0,v_0)$ and $(u_0,(\rho_a/|v_0|) v_0)$ is contained in
$p(F'_\Phi)\cap (\T_a\setminus T'_a)$. From this, one quickly observes that
the inclusion $p(F'_\Phi)\cap (\T_a\setminus T'_a)\to \T_a\setminus T'_a$
is isotopic to the inclusion of $(\T_a\setminus T'_a)\setminus T'$ fixing
a neighborhood around both $F'_f\cap (\T_a\setminus T'_a)$ and 
$\{|v| = \rho_a\}$.

Finally, all these diffeomorphisms glue together to the desired map.

For the monodromy, we notice that the diffeomorphism type of the pair
$(F'_{\Phi,\theta},p^{-1}(\T_2\setminus (T_1\cup T_g)) \cap F'_{\Phi,\theta})$,
where $\theta \in S^1$ and $F'_{\Phi,\theta} = \Phi'^{-1}(\theta \epsilon)$,
is independent of $\theta$, that is,
$p^{-1}(\T_2\setminus (T_1\cup T_g)) \cap F'_{\Phi,\theta}$ is a subbundle
of the Milnor fibration. The description of this fibre above is independent
of $\theta$ however, and therefore gives a trivialisation of the bundle.
Therefore, the monodromy acts trivially, up to homotopy, on this subset.
\end{proof}

\begin{proof}[Proof of \ref{th:reform_strict_g}]
Let $a\in\A_g$. As before, we consider coordinates $u,v$ on $\T_a$ so that
$f'=u^m$ and $g'=u^l v^k$. Then $H_a := F'_\Phi\cap p^{-1}(\T'_a)$ is the set
of points $(u,v,z)$ satisfying $|z| \leq \delta$, $|v| \leq \eta$ for some
$\eta \ll \epsilon$ and the equality $\Phi' = \epsilon$.
We show first that abstractly, this set is a disjoint union of bidisks.
Clearly, the map $\pi_a = (v,z):H_a \to D_\eta \times D_\delta$ is a proper
surjection
which maps boundary points to boundary points. Also, the preimage of $(0,0)$
is the set $\{(u,0,0):u^m = \epsilon\}$, and so contains exactly $m$ points.
By the implicit function theorem, if $\partial_u \Phi' \neq 0$ on $H$, the map
$\pi_a$ is a local diffeomorphism, and so a covering map. Furthermore,
since the bidisk is contractible, such a covering map must be a product. We get
\[
  \partial \Phi'(u,v,z) = \partial_u(u^m + zu^l v^k) = m u^{m-1} + zlu^{l-1}v^k.
\]
The functions $|m u^{m-1}|$ is bounded below by a positive number on $H_a$,
since it is continuous and does not vanish. Similarly, the function
$|zlu^{l-1}|$ is bounded above. Taking $\eta$ small enough, we get
$|m u^{m-1}| > |zlu^{l-1}v^k|$ on $H_a$. This gives $\partial_u\Phi'\neq 0$
as required.

We have now shown that $F'_\Phi$ is given by glueing handles (such as $h$) to
$T_{f,g}\setminus T_g'$ in the way described in
\ref{pre:framings}. We only have to determine the twisting coefficient.
We already have a parametrization of the handle $h$ by $(u,v)$. 
The handle already contained in $T_{f,g}$ is parametrized by $(u-\xi,v)$,
where $\xi$ is some $m$-th root of unity. Denote
this parametrization by $\psi:\bar D\times\bar D\to T_{f,g}$.

Now, for any $r\in h$ with coordinates $(z,v)$ we have
$p(r) = (U(z,v), v)$ where
\[
  z U(z,v)^l v^k = U(z,v)^m - \epsilon
\]
and we assume that $U(z,v)$ is in a small neighbourhood around some $m$-th
root of $\epsilon$. This shows that the twisting coefficient used to glue
$h$ is $k$, as stated.

To finish the proof, we must consider the action of the monodromy on the
handles corresponding to $a\in\A_{g,2}$. But the central disks of these
handles are given by $F'_f$. It follows that they are permuted cyclically.
\end{proof}

\footnotesize
\noindent\textit{Acknowledgments.}
The author is supported by the PhD program of the CEU, Budapest and
by the `Lend\"ulet' and ERC program `LTDBud' at R\'enyi Institute.

\bibliographystyle{plain}
\bibliography{bibliography}
\end{document}